\documentclass[11pt]{article}
\usepackage{graphicx}
\usepackage{amsmath}
\usepackage{amsfonts}
\usepackage{amsthm}
\usepackage{amssymb}
\usepackage{color}

\usepackage{verbatim} 
\usepackage[hyphens]{url}
\usepackage{hyperref}
\usepackage[margin=1in]{geometry}
\usepackage{setspace}
\usepackage{float}
\usepackage{color}
\newtheorem{theorem}{Theorem}
\newtheorem{lemma}[theorem]{Lemma}
\newtheorem{corollary}[theorem]{Corollary}

\begin{document}

\title{Weighted Mediants and Fractals}

\author{Dhroova Aiylam \\ MIT \\ \href{mailto:dhroova@mit.edu}{dhroova@mit.edu} \and Tanya Khovanova \\ MIT \\ \href{mailto:tanyakh@yahoo.com}{tanyakh@yahoo.com}}

\maketitle

\begin{abstract}
In this paper we study a natural generalization of the Stern-Brocot sequences which comes from the introduction of weighted mediants. We focus our attention on the case $k = 3$, in which $(2a + c)/(2b + d)$ and $(a + 2c)/(b + 2d)$ are the two mediants inserted between $a/b$ and $c/d$. We state and prove several properties about the cross-differences of Stern-Brocot sequences with $k = 3$, and give a proof of the fractal-like rule that describes the cross-differences of the unit $k = 3$ Stern-Brocot sequences, i.e. the one with usual starting terms $0/1, 1/1$ and with reduction of fractions.
\end{abstract}

\section{Introduction}

\indent

The Stern--Brocot tree is an object of classical interest in number theory. Discovered independently by Moritz Stern in 1858 \cite{Stern} and Achille Brocot in 1861 \cite{Brocot}, it was originally used as a way to find rational approximations of certain kinds to specific numbers. As a consequence, the Stern-Brocot tree is deeply connected to the theory of continued fractions. It also comes up in a variety of other contexts, including Farey Sequences, Ford Circles, and Hurwitz' theorem.

The classical Stern-Brocot sequences are generated row by row, as follows: the first row has entries $\frac{0}{1}$ and $\frac{1}{0}$. In each subsequent row, all entries from the previous row are copied and between every pair of neighboring entries $\frac{a}{b}$ and $\frac{c}{d}$ the mediant fraction $\frac{a + c}{b + d}$ is inserted. This process is repeated ad infinitum; the result gives the Stern-Brocot sequences as rows.

One can generalize the notion of a mediant if we assign integer weights to fractions. Suppose we assign weight $m$ to the left fraction $a/b$ and weight $n$ to the right fraction $c/d$. Then the weighted mediant is $(ma+nc)/(mb+nd)$. The total weight in this case is $m+n$, so that the classical case corresponds to a total weight of 2 with $m=n=1$.

James Propp proposed a weighted generalization of the Stern-Brocot sequences \cite{Propp} via a modified mediant-insertion process. Instead of a single mediant fraction, all possible weighted mediants of total weight $k$ are inserted, in increasing order, between every pair of consecutive fractions in the previous row. The number of inserted mediant fractions is $k-1$; thus the classical Stern-Brocot sequences correspond to the case $k = 2$, while we focus our attention here on the case $k = 3$.

Our main results in this paper are the following: we state and prove several properties of the cross-differences of Stern-Brocot sequences with $k = 3$, culminating in a proof of the fractal-like rule that governs the cross-differences in Stern-Brocot sequences corresponding to the starting terms $0/1$ and $1/1$, and with reduction of fractions.

In Section~\ref{sec:defs}, we define precisely the Stern-Brocot  sequences from weighted mediants given a pair of starting terms. We also define the cross-difference of two consecutive fractions $a/b$ and $c/d$ in the same row as $bc-ad$. We discuss the properties of cross-differences.

In Section~\ref{sec:RandCDs} we explain how reduction to lowest terms and cross-differences are related to each other. In Section~\ref{sec:noRed} we suggest a variation where the fractions are never reduced to lowest terms. In this variation the formula for cross-differences is simple and has a fractal structure.

In Section~\ref{sec:utRed} we go back to our main object of study, the unit Stern-Brocot sequences of weight 3 and show how the reduction works in this case. Understanding the reduction allows us to make a careful analysis of the cross-differences in the most interesting case, with starting terms $0/1$ and $1/1$ and in which fractions are reduced. This happens in Section~\ref{sec:CDs} where we ultimately use our analysis to prove the fractal-like rule which the cross-differences follow. The cross-differences of one Stern-Brocot sequence exhibit some self-similar behavior, but they are not completely self-similar. We call this structure \textit{quasi-fractal}.

We continue describing our quasi-fractals in Section~\ref{sec:CDsCont}, where we give an explicit formula for the cross-differences. We conclude in Section~\ref{sec:reallife} with examples from real life where quasi-fractals similar to ours appear.

\section{Stern-Brocot sequences from weighted mediants. Notation and definitions}\label{sec:defs}

For a fixed parameter $k$, we say the weighted mediants of two fractions $a/b$ and $c/d$ are $$\frac{(k - 1)a + c}{(k - 1)b + d}, \; \; \frac{(k - 2)a + 2c}{(k - 2)b + 2d}, \; \; \dots, \; \; \frac{a + (k - 1)c}{b + (k - 1)d}$$ whence there are $k - 1$ mediants in all. As in the classical Stern-Brocot case, we start with two terms and each row is obtained by inserting mediants between consecutive fractions in the previous row.  With this notation, the classical Stern-Brocot corresponds to $k = 2$ with starting terms $0/1$ and $1/0$. The second row is $0/1$, $1/1$, and $1/0$. The sequences can be divided into two equivalent halves by the mid-line. Indeed, if we swap numerators and denominators and reverse the order, the left half becomes the right half. For this reason many researchers study only the left half of, that is, the sequences with the starting numbers $0/1$ and $1/1$. The other rational starting points for the classical Stern-Brocot sequences were studied in \cite{A}.

In this paper we restrict our attention to the case $k = 3$ with starting terms $0/1$ and $1/1$. We call this case the \textit{unit} case. Here is what the first three rows of the unit case look like:

	\[ \frac{0}{1} \; \; \; \frac{1}{1}\]  \[\frac{0}{1} \; \; \; \frac{1}{3} \; \; \; \frac{2}{3} \; \; \; \frac{1}{1}\] 
	\[\frac{0}{1} \; \; \; \frac{1}{5} \; \; \;  \frac{2}{7} \; \; \; \frac{1}{3} \; \; \; \frac{4}{9} \; \; \; \frac{5}{9} \; \; \; \frac{2}{3} \; \; \; \frac{5}{7} \; \; \; \frac{4}{5} \; \; \; \frac{1}{1}\]
	
Let us now introduce some notation and definitions. If $\frac{p}{q}$ and $\frac{r}{s}$ are rational numbers in lowest terms, their \emph{weighted mediants} are the numbers $\frac{2p + r}{2q + s}$ and $\frac{p + 2r}{q + 2s}$ in lowest terms. We call these the \textit{left} and \textit{right} mediants of $\frac{p}{q}$ and $\frac{r}{s}$, respectively. We say that $\frac{p}{q}$ and $\frac{r}{s}$ are the \textit{parents} of the mediants $\frac{2p + r}{2q + s}$ and $\frac{p + 2r}{q + 2s}$. Notice that numbers in each row are in increasing order.

Next let $SB_i$ stand for the $i$-th row. By tradition, the first row---the starting terms  $\frac{0}{1}$ and $\frac{1}{1}$ are considered row 0.
The sequence $SB_i$ is also called \textit{the Stern-Brocot sequence of order $i$}. Thus $SB_0 = \{ \frac{0}{1}, \frac{1}{1}\}$ and $SB_{i + 1}$ is obtained by copying all terms from $SB_{i}$ and inserting between every pair of consecutive fractions $\frac{p}{q}, \frac{r}{s} \in SB_{i}$ their weighted mediants.

We say the $\emph{cross-difference}$ of two fractions $\frac{p}{q}$ and $\frac{r}{s}$ is $\mathcal{C}(\frac{p}{q}, \frac{r}{s}) = qr - ps$. We are most interested in the cross-differences of consecutive numbers in $SB_i$. As was the case  when $k = 2$ \cite{AK}, the cross-difference essentially determines how fractions in the Stern-Brocot sequences are capable of reducing. In particular, the factor by each the ratio of a weighted mediant is reduced to its lower terms is a factor of $\mathcal{C}(\frac{p}{q}, \frac{r}{s})$, as we prove in Lemma~\ref{thm:red-cd}.

The cross-difference of two fractions is positive if the second fraction is larger than the first. In our case all the cross-differences are positive.

It is important to remember that the cross-difference depends on the representation of rational numbers, not just on the numbers themselves. In particular, when we reduce one of the fractions the value of the cross-difference decreases, and so the cross-difference is smallest when both rational numbers are in lowest terms. 

The following statements describe known results about 3-Stern-Brocot sequences \cite{AK}.

\begin{lemma}
All the denominators in the Stern-Brocot sequence with $k=3$ are odd and the numerators in each row alternate between even and odd.
\end{lemma}

\begin{corollary}
All cross-differences are odd.
\end{corollary}

\begin{lemma}
The number of terms in $SB_n$ is $3^{n} +1$.
\end{lemma}

Finally we state the theorem about rational numbers that appear in the Stern-Brocot sequences.

\begin{theorem}
All the rational numbers between $0$ and $1$ such that their representation in the lowest terms has an odd denominator appear in the unit Stern-Brocot sequences.
\end{theorem}

\section{Reduction and Cross-Differences}\label{sec:RandCDs}

In the classical $k=2$ Stern-Brocot case, the cross-difference of the initial terms is 1. It follows that mediants are always in lowest terms. 

In the $k=3$ case, this is no longer true. For example, the second row of $SB(\frac{0}{1}, \frac{1}{1})$  contains two consecutive entries $1/3$ and $4/9$. Their weighted mediants before reduction are: $6/15$ and $9/21$. They both are reduced by 3 and the new entries in the third row are $2/5$ and $3/7$.

There is a simple but important connection between the reduction and cross-differences.

\begin{lemma}\label{thm:red-cd}
The reduction factor of the left/right mediants of two fractions $\frac{a}{b}$ and $ \frac{c}{d}$ divides their cross-difference $\mathcal{C}\left(\frac{a}{b},\frac{c}{d}\right) = bc-ad$.
\end{lemma}

\begin{proof}
Suppose the left mediant $\frac{2a + c}{2b + d}$ is reducible. That is, $2a+c = gp$, and $2b+d = gq$, where $g = \gcd(2a+c,2b+d)$ is the reduction factor. Multiplying the first equality by $b$ and the second by $a$ and subtracting them, we get
$$bc -da = 2ab+cb-2ba-da = gpb-gqa = g(pb-qa).$$
That is $g$ divides $bc-ad$. Similar reasoning shows the same is true for the right mediant.
\end{proof}

From now we will focus not on the rows of the  Stern-Brocot sequences, but rather on the rows of cross-differences.

\section{No Reduction}\label{sec:noRed}

Let us first consider a simple question: what are the cross-differences of the Stern-Brocot sequences which are obtained from the starting terms $\frac{0}{1}$ and $\frac{1}{1}$, except where we do not reduce fractions? In this case, consecutive fractions $\frac{a}{b}$ and $\frac{c}{d}$ in any row become the fractions $\frac{a}{b}$, $\frac{2a + c}{2b + d}$, $\frac{a + 2c}{b + 2d}$, and $\frac{c}{d}$ in the next row. Since these fractions are never reduced, their pairwise consecutive cross-differences are $(bc - ad)$, $3(bc - ad)$, and $(bc - ad)$. Thus the rows of consecutive cross-differences evolve according to a simple propagation rule:

\noindent
\textbf{No-Reduction Propagation Rule.} An instance of $C$ in row $r$ becomes $C$, $3C$, $C$ in row $r + 1$. 

The single cross-difference in row zero is $1$, so in the no-reduction case all cross-differences are powers of $3$. For instance, the cross-differences in the first row are: $1$, $3$, $1$ and the cross-differences in the second row are: $1$, $3$, $1$, $3$, $9$, $3$, $1$, $3$, $1$. 

To make it easy to visualize the rows of the cross-differences, we take the base-$3$ logarithm of every cross-difference and present them as a graph. Figure~\ref{fig:1and2} shows the first and the second row; the graphs have been rescaled to the same size to emphasize the differences in shape. 

\begin{figure}[H]
\centering
$\begin{array}{cc}
\includegraphics[scale=0.3]{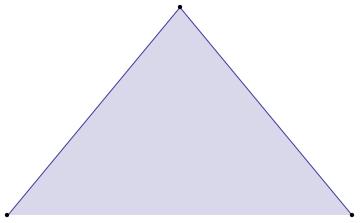} &
\includegraphics[scale=0.3]{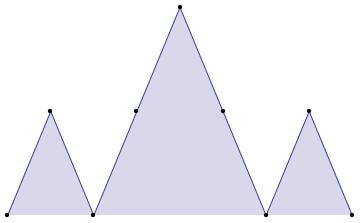}
\end{array}$
\caption{The first and the second row of cross-differences}\label{fig:1and2}
\end{figure}

Figure~\ref{fig:3and4NoRed} shows the third and fourth row. It is easy to see how each figure is generated from the previous: the left third and  right third of the graph are copies of the previous row, while the middle third is the same copy moved up by 1. The graphs exhibit a self-similar, fractal-like structure. 

\begin{figure}[H]
\centering
$\begin{array}{cc}
\includegraphics[scale=0.35]{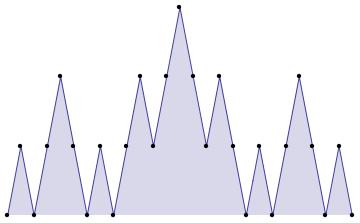} &
\includegraphics[scale=0.35]{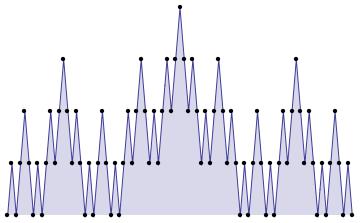}
\end{array}$
\caption{The third and the fourth row of cross-differences}\label{fig:3and4NoRed}
\end{figure}

We can explicitly specify the value at the $i^{\text{th}}$ index in each row of cross-differences, where indexing begins at $0$:

\begin{lemma}\label{lemma:ternary}
In the no-reduction case, the $i$-th cross-difference is $3^{n(i)}$ where $n(i)$ is the number of $1$s in the ternary expansion of $i$.
\end{lemma}

\begin{proof}
We argue by induction. The base case of the is trivial, so let us suppose the claim holds for row $n$. By the propagation rule above, the values at indices $3i$, $3i + 1$, and $3i + 2$ of row $(n + 1)$ are $C$, $3C$, and $C$ respectively, where $C$ is the value at index $i$ in row $n$. The induction step now follows from the fact that $n(3i) = n(3i + 2) = n(i)$ whereas $n(3i + 1) = n(i) + 1$.
\end{proof}

We would like to emphasize two properties of the cross-differences of the sequences without reduction that will survive the transition to reduction.

\textbf{Property 1.} The value of the cross-difference is the same for all the values of $i$ with the same set of ones in the ternary expression.

\textbf{Property 2.} The values of the cross-differences at index $i$ is the same for every row that contains this index.

The second property allows us to view the row of cross-differences as a single infinite sequence which is the union of all the sequences.

Before proceeding to the reduction case, let us examine another graphical representation of cross-differences. We  divide the interval $[0,1]$ into $3^i$ equal intervals and define a piece-wise constant function which, on the $i^{\text{th}}$ interval, is equal to the base-$3$ logarithm of the $i$-th cross-difference. For example, the second row corresponds to the function in Figure~\ref{fig:NoRedFlat2}.

\begin{figure}[htbp]
\centering
\includegraphics[width=5cm, height=2cm]{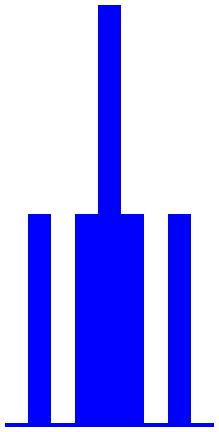}
\caption{A piece-wise constant representation of the second row.}\label{fig:NoRedFlat2}
\end{figure}

Observe that the intervals where the function vanishes on the $i^{\text{th}}$ row form the $i^{\text{th}}$ iteration of the Cantor set. The first several iterations of the Cantor set are shown in Figure~\ref{fig:CantorSet}.
\newline

\begin{figure}[htbp]
\centering
\includegraphics[width=5cm, height=1cm]{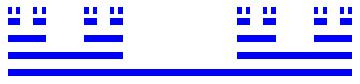}
\caption{Iterations of the Cantor Set.}\label{fig:CantorSet}
\end{figure}

\section{Reduction}\label{sec:utRed}

We now return to the case with reduction. When a fraction is reduced by a factor of $m$, both cross-differences in which it participates are divided by $m$. It follows by induction that even in the case where fractions are reduced, all cross-differences are powers of $3$.

Our ultimate goal is to give an explicit description (as in Lemma~\ref{lemma:ternary}) which characterizes the cross-differences in Stern-Brocot sequences with reduction. To do so, it will be critical to understand precisely where, and by what factor, fractions are reduced. 

To that end, let us consider the Stern-Brocot sequences modulo $9$. More formally, we replace each fraction $a/b$ with $x/y$, where $x$, $y$ are the residues of $a$, $b$ respectively modulo $9$. For every fraction $a/b \pmod{9}$ which occurs in the Stern-Brocot sequences, there are $17$ distinct possibilities modulo $9$ for the fraction which follows it:

\begin{enumerate}
    \item The cross-difference is 1 modulo 9. There are 9 such pairs.
    \item The cross-difference is 3 modulo 9. There are 6 such pairs.
    \item The cross-difference is 0 modulo 9, in which case $c/d= a/b$ modulo 9 or $c/d=(9-a)/(9-b)$ modulo 9.
\end{enumerate}

The third point is noteworthy since it means that not all possible fractions $c/d \pmod{9}$ with $(bc - ad) \equiv 0 \pmod{9}$ can follow $a/b$ in a Stern-Brocot sequence.

\begin{lemma}\label{thm:mod9}
If the cross-difference of two consecutive fractions $a/b$ and $c/d$ is divisible by 9, then either $(c-a)/(d-b)= 0/0$ modulo 9 or $(c+a)/(d+b)= 0/0$ modulo 9.
\end{lemma}

\begin{proof}
There are two ways for consecutive fractions $a/b$ and $c/d$ with  cross-difference divisible by 9 to appear in a Stern-Brocot sequence:
\begin{enumerate}
    \item They are the left and the right mediant of two fractions with cross-difference equal to $3$, and neither was reduced.
    \item The cross-difference of their parents is divisible by 9.
\end{enumerate}

\textbf{Case 1:} Suppose $p/q$ and $r/s$ are the parents of $a/b$, $c/d$. If $p$ or $r$ is divisible by $3$, then the other must be as well. Since there was no reduction, this means $a + c = (2p + r) + (p + 2r) = 3(p + r)$ is divisible by $9$. Now $q$ and $s$ must leave distinct (non-zero) residues modulo $3$, for otherwise $b = 2q + s$ and $d = q + 2s$ would be divisible by $3$ and the mediants would reduce. Then $q + s \equiv 0 \pmod{3}$, so that $b + d = (2q + s) + (q + 2s) = 3(q + s)$ is divisible by $9$ as well. Thus the claim holds in this case, and analogous reasoning applies to the case where $q$ or $s$ is divisible by $3$. 

Hence we can assume none of $p$, $q$, $r$, $s$ is divisible by $3$. As before, we only need to check one representative modulo 3. There are 16 possibilities for $(p, q, r, s) \pmod{3}$. Up to  symmetry -- we can swap $p/q$ with $r/s$, and also swap numerators with denominators -- it therefore suffices to consider the following cases:
$(1,1,1,1)$, $(1,1,1,2)$, $(1,1,2,2)$,
$(1,2,1,2)$,
$(1,2,2,1)$,
$(1,2,2,2)$, and $(2,2,2,2)$. However, we can exclude cases when $p/q \equiv r/s \pmod{3}$ since these cases result in a reduction. We can also exclude cases when the  cross-difference is not divisible by 3. We are left with just two possibilities: $(1,1,2,2)$ and 
$(1,2,2,1)$.

Both of these possibilities have $p + r \equiv q + s \equiv 0 \pmod{3}$. Then $a/b = (2p + r)/(2q + s)$ and $c/d = (p + 2r)/(q + 2s)$, whence $a + c = (2p + r) + (p + 2r) = 3(p + r)$ and $b + d = (2q + s) + (s + 2q) = 3(q + s)$ are both divisible by $9$. Thus the claim holds in this case. 

\textbf{Case 2:} In this case we cannot assume that $a/b$ and $c/d$ are the left and the right mediant of some pair of fractions in the previous row, since one of them might have been in the previous row. However, it still makes sense to speak of the parents of $a/b$ and $c/d$, where we simply mean the pair of fractions in the previous row between which $a/b$ and $c/d$ lie (inclusive at either endpoint). 

Suppose the parents' cross-difference is divisible by 9. By induction on the row number we can assume the lemma holds in the previous row, so that the parents $p/q$ and $r/s$ either have equal or  complementary remainders modulo 9. First suppose $p/q \equiv r/s \pmod{9}$; then the left and the right mediants will be reduced by a factor of exactly $3$. Indeed, we have $2p + r  \equiv p + 2r \equiv 0 \pmod{3}$ and similarly $2q + s \equiv q + 2s \equiv 0 \pmod{3}$ so the mediants are reduced by a factor of at least $3$. On the other hand, we need $3 | p, q, r, s$ in order to reduce by a factor of $9$, which contradicts the fact that $p/q$, $r/s$ are reduced fractions. 

Then in order for $9 | (bc - ad)$ we must in fact have $27 | (qr - ps)$. It's not hard to check now that numbers in the next row are all equal to of $p/q$ modulo $9$. Thus every pair of consecutive numbers is the same modulo $9$.

Suppose instead that the parents $p/q$ and $r/s$ are complementary modulo $9$, i.e. $(p+r)/(q+s) = 0/0 \pmod{9}$. Then the left mediant and right mediant modulo $9$ are $p/q$ and $r/s$ respectively, so there is no reduction. The four fractions in the next row are $p/q$, $p/q$, $r/s$, and $r/s$ modulo $9$. Every consecutive pair here is either the same or complementary modulo $9$, as required.
\end{proof}

As part of the proof of Lemma~\ref{thm:mod9} we established the following fact:

\begin{corollary}\label{thm:redcd9}
If the cross-difference of two fractions is divisible by 9 and they are the same modulo 9, then both new mediants are reduced by 3. If the fractions are complementary modulo 9, then there is no reduction in new mediants.
\end{corollary}

We are now ready to prove the main theorem about reduction.

\begin{theorem}\label{thm:red1and3}
When fractions in the Stern-Brocot sequences reduce non-trivially, they do so by a factor of exactly $3$. 
\end{theorem}

\begin{proof}
In order for a fraction to reduce by a higher power of $3$, we require the parents' cross-difference be divisible by $9$. Indeed, consider the left mediant $(2a+c)/(2b+d)$. If 9 divides both $2a+c$ and $2b+d$, then it divides their linear combination: $b(2a+c) - a(2b+d) = bc-ad$. The theorem now is the consequence of Corollary~\ref{thm:redcd9}.
\end{proof}

As a consequence, reduction in the Stern-Brocot sequences is symmetric.

\begin{lemma}
The left mediant reduces by the same factor as the right mediant reduces.
\end{lemma}

\begin{proof}
We saw before that the reduction happens precisely when $a/b$ and $c/d$ have the same remainders modulo $3$. Then the left and right mediant will always reduce at the same time, and by Theorem~\ref{thm:red1and3} they reduce by the same factor.
\end{proof}

Theorem~\ref{thm:red1and3} tells us something about the values we see in Stern-Brocot sequences. Suppose $a/b$ and $c/d$ are consecutive fractions in row $i$. Then if there is no reduction, both new mediants have numerators that are more than $\max(a,c)$ and, similarly, denominators that are more than $\max(b, d)$. If there was a reduction and new mediants are $x/y$ and $z/w$, then the numerators are consecutive: either increasing $a < x < z < b$, or decreasing $a > x > z > b$. The analogous statement for denominators holds as well.

\section{Cross-differences}\label{sec:CDs}

We continue our discussion of cross-differences in the Stern-Brocot sequences with reduction. 

Denote by $C_{n}$ the ordered list of cross-differences of adjacent pairs in the Stern-Brocot sequence of order $n$, so that $C_0 = \{1\}$, $C_1 = \{1, 3, 1\}$, and $C_2 = \{1, 3, 1, 3, 9, 3, 1, 3, 1\}$. We see that these two rows are the same as the rows without reduction (see Figure~\ref{fig:1and2}). 

Starting from the next row, a different picture emerges. Figure~\ref{fig:3Red} shows the base-$3$ logarithm of $C_3 = \{1, 3, 1, 3, 9, 3, 1, 3, 1, 1, 1, 1, 9, 27, 9, 1, 1, 1, 1, 3, 1, 3, 9,
3, 1, 3, 1\}.$

\begin{figure}[htbp]
\centering
\includegraphics[scale=0.35]{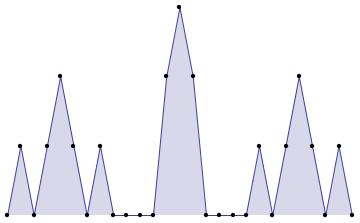}
\caption{The third row of cross-differences}\label{fig:3Red}
\end{figure}

Figures~\ref{fig:4Red} and~\ref{fig:5Red} show $C_4$ and $C_5$ respectively.

\begin{figure}[htbp]
\centering
\includegraphics[scale=0.4]{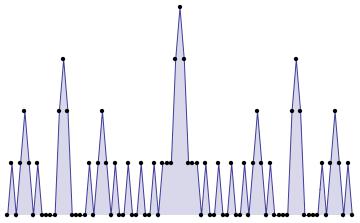}
\caption{The fourth row of cross-differences}\label{fig:4Red}
\end{figure}

\begin{figure}[htbp]
\centering
\includegraphics[scale=0.5]{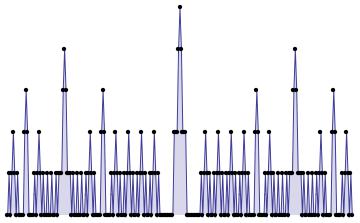}
\caption{The fifth row of cross-differences}\label{fig:5Red}
\end{figure}

We see that each new picture is divided into three parts. The first part and the last part are copies of the previous picture. We will prove this later, but for now we simply note the similarity to the no-reduction case. The middle part is some adjustment of the previous picture. These pictures look fractal-like, but are not quite fractals. We call them \textit{quasi-fractals}.

In this section we will give, with proof, a rule which describes how to get from $C_n$ to $C_{n + 1}$. It is natural to think of replacing each value in $C_n$ with three values to obtain $C_{n + 1}$, since a value in $C_n$ is the cross-difference of two consecutive fractions $a/b$ and $c/d$, which turn into four fractions (and hence, three cross-differences) when the two mediants are inserted in row $n + 1$. The rule to go from one value in $C_n$ to the corresponding three values in $C_{n + 1}$ is quite simple, and given by the following lemma.

\begin{lemma}\label{localpropagation}
The cross-difference $V$ in one row is replaced in the next row by either $(V,3V,V)$ or $(V/3,V/3,V/3)$.
\end{lemma}

\begin{proof}
The first case occurs when mediants do not reduce, the second when they do.
\end{proof}

It remains to determine precisely when each rule applies. Let us denote the $i^{\text{th}}$ entry in row $C_{n}$ as $C_n(i)$. At least one of the cases is straightforward:

\begin{lemma}\label{fractal1}
If $C_n(i) = 1$, then $$(C_{n + 1}(3i), C_{n + 1}(3i + 1), C_{n + 1}(3i + 2)) = (1, 3, 1).$$
\end{lemma}

\begin{proof}
If $C_n(i)=1$, the mediants clearly cannot be reduced.
\end{proof}

For other cases we need to look in some surrounding neighborhood.

\begin{lemma}\label{fractal2}
If $C_n(i)$ is a strict local maximum, then $$(C_{n + 1}(3i), C_{n + 1}(3i + 1), C_{n + 1}(3i + 2)) = (C_n(i), 3C_n(i), C_n(i)).$$
\end{lemma}

\begin{proof}
A local maximum can only occur as the cross-difference between a left and  right mediant. Indeed, when a fraction is copied to the next row its cross-differences are either equal to the neighboring cross-differences if there was a reduction, or smaller if there was no reduction.

Hence the fractions $a/b$ and $c/d$ corresponding to $C_n(i)$ are the left and right mediants of the same parents, and they must not have been reduced. Suppose the parents of $a/b$ and $c/d$ are $p/q$ and $r/s$. We must have either $p \not \equiv r \pmod{3}$ or $q \not \equiv s \pmod{3}$ since $a/b$ and $c/d$ were not reduced, and so we can compute the mediants of $a/b$ and $c/d$ exactly as $(5p + 4r)/(5q + 4s)$ and $(4p + 5r)/(4q + 5s)$. Since $p \equiv r \pmod{3}$ and $q \equiv s \pmod{3}$ do not both hold, these mediants are irreducible as claimed.
\end{proof}

It turns out these are the only cases where the first rule applies.

\begin{lemma}\label{fractal3}
If $C_n(i) \neq 1$ nor $C_n(i)$ is a strict local maximum, then $$(C_{n + 1}(3i), C_{n + 1}(3i + 1), C_{n + 1}(3i + 2)) = (C_n(i)/3, C_n(i)/3, C_n(i)/3).$$
\end{lemma}

\begin{proof}
Suppose $C_n(i)$ corresponds to the consecutive fractions $a/b$ and $c/d$. The fact that $C_n(i)$ is not a local maximum means that either one of the corresponding fractions was copied from the previous row, or that both of them are mediants that were reduced.

In the \textbf{first case}, we can assume without loss of generality that $a/b$ was copied from the previous row. Denote by $z/w$ the other parent of $c/d$. Either $c/d$ was not reduced, so that $c = (2a + z)$ and $d = (2b + w)$, or else $c/d$ was reduced, in which case $c = (2a + z)/3$ and $d = (2b + w)/3$. We consider these cases separately. 

First suppose $c/d$ is not reduced. Then either $a \not \equiv z$ or $b \not \equiv w \pmod{3}$, while the theorem statement requires $3 | bc - ad$ whence $3 | bz - aw$. We can now do casework on $a, b, z, w$ modulo $3$. If neither $a$ nor $b$ is divisible by $3$, then neither $z$ nor $w$ is divisible by 3. The divisibility properties of the mediants will not change if we multiply both $a$ and $z$ by a number not divisible by 3; the same is true for $b$ and $z$, so without loss of generality we may assume that $a, b \equiv 1 \pmod{3}$. Now if $z \equiv 1 \pmod{3}$, then $3 | bz - aw$ forces $w \equiv 1 \pmod{3}$ which contradicts the fact that $c/d$ does not reduce. Otherwise, if $z \equiv 2 \pmod{3}$ then $3 | bz - aw$ forces $w \equiv 2 \pmod{3}$ and the mediants of $a/b$ and $c/d$ reduce, as claimed. This leaves only the case when exactly one of $a, b$ is divisible by $3$; without loss of generality we assume it is $a$. Then $3 | bz - aw$ forces $3 | z$, and we can assume (as before) without loss of generality that $b \equiv 1 \pmod{3}$. If $w \equiv 1 \pmod{3}$, we contradict the fact that $c/d$ does not reduce. If instead $w \equiv 2 \pmod{3}$, then the mediants of $a/b$ and $c/d$ reduce as claimed.

Suppose instead that $c/d$ is reduced. The theorem statement requires $3 | bc - ad$, so since $c/d$ is reduced we must have $9 | bz - aw$. It now follows from Corollary~\ref{thm:redcd9} that $a \equiv z \bmod{9}$ and $b \equiv w \bmod{9}$. It's now easy to check that the mediants of $a/b$ and $c/d$ reduce: we have $c = (2a + z)/3$ and $d = (2b + w)/3$, so $$2a + c = 2a + (2a + z)/3 \equiv 3a \equiv 0 \bmod{3}.$$ Similarly, $$2b + d = 2b + (2b + w)/3 \equiv 3b \equiv 0 \bmod{9}.$$ Thus the claim holds in the \textbf{first case}.

Now consider the \textbf{second case}, in which $a/b$ and $c/d$ are the reduced left and right mediant of the same parents. That means the cross-difference of their parents $p/q$ and $r/s$ is divisible by 9. We can now apply Corollary~\ref{thm:redcd9} to conclude that $c/d = a/b$ modulo 9, whence their mediants reduce as well. 
\end{proof}

Combining all the cases we get our main theorem on how the cross-differences propagate.

\begin{theorem}\label{fractal}
If $C_n(i) = 1$ or $C_n(i)$ is a strict local maxima, then $$(C_{n + 1}(3i), C_{n + 1}(3i + 1), C_{n + 1}(3i + 2)) = (C_n(i), 3C_n(i), C_n(i)).$$ Otherwise, $$(C_{n + 1}(3i), C_{n + 1}(3i + 1), C_{n + 1}(3i + 2)) = (C_n(i)/3, C_n(i)/3, C_n(i)/3).$$
\end{theorem}

\section{Cross-Differences Continued}\label{sec:CDsCont}

Here we want to look at different cool properties of the rows of the cross-differences.

\begin{lemma}
Suppose $C_n(i) = 1$. Then the values in row $n+m$ with the indices between $3^mi$ and $3^m(i+1)-1$ inclusive are the copy of $C_m$. In other words, $C_{n+m}(3^m i +k) = C_m(k)$, for $ 0 \leq k < 3^m$.
\end{lemma}

\begin{proof}
If some range in the row of cross-differences starts and ends with 1, the result of repeated  propagation does not depend on the neighbors. Thus if we start with just $1$, the result of $m$ consecutive propagation is $C_m$. 
\end{proof}

We promised this statement before: the proof is immediate now.

\begin{corollary}
The first and the last thirds of $C_n$ are copies of $C_{n-1}$.
\end{corollary}

\begin{proof}
The first/last third of $C_n$ are the result of the propagation of the first/last 1 in $C_1$.
\end{proof}

This fact is similar to Property~2 for the sequences without reduction, in the sense that it allows us to view all the rows of cross-differences as a single infinite sequence.

What happens to the middle third? Notice that $C_3$ has three ones in a row in a place that we can call the first third of the middle third. It follows that $C_{n+3}$ will have three copies of $C_n$ in the first third of the middle third.

In general, we see copies of previous rows on the outskirts of a given row. The new things happen in the very middle of a row. In order to better describe this middle behavior we introduce the notion of a steeple.

\subsection{Steeples}

The \textit{steeple} is defined as the largest range of values in a row of cross-differences containing the middle and not containing ones.

In Figure~\ref{fig:steeples} we show the logarithms of the steeples in the first eight rows in sequence, separated by zeroes. From the picture we can see that the even-indexed logarithmic steeples can be obtained from the previous steeple via a shift up by 1. The odd-indexed logarithmic steeples have a middle third that is the same as the previous steeple shifted up by $1$, while the first and last third consist of ones.

\begin{figure}[htbp]
\centering
\includegraphics[scale=0.5]{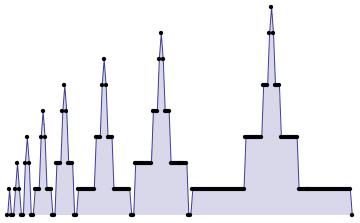}
\caption{Steeples}\label{fig:steeples}
\end{figure}

The following lemma gives a full description of steeples. First, we define $m(i)$ as the index of the first zero or two in the ternary representation of $i$ padded to have $n$ digits. We begin our indexing at $1$, and call this number the \textit{middleness} of the index $i$. For the middle index $i = (3^n-1)/2$, the middleness is not defined. Otherwise, the smaller $m(i)$ is, the further away from the middle the index $i$ is.

\begin{lemma}\label{thm:mainsteeple} If   $m = m(i) > \lceil n/2 \rceil$, then $i$ corresponds to a steeple. Moreover,  $C_n(i) = 3^{2m-n-1}$. If $i$ is the middle point, then $C_n(i) = 3^{n}$.
\end{lemma}

\begin{proof}
 We proceed by induction. The base case is clear, so let us suppose that the lemma holds for the $n$-th steeple.

Notice that the only peak in the steeple is the middle point. By the propagation rule, the middle point in row $n$ becomes three points in row $n+1$ with values $(3^n, 3^{n+1}, 3^n)$ and indices $(3^{n+1}-3)/2,(3^{n+1}-1)/2,(3^{n+1}+1)/2$. The middleness of the indices $(3^{n+1}-3)/2$ and $(3^{n+1}+1)/2$ is $n+1$. This matches the formula.

On the other hand, consider index $i$ in row $n$ that does not correspond to the middle. It propagates to three points with indices $3i$, $3i+1$, and $3i+2$ in row $n+1$. That is, the new three indices have the same first $n$ digits in their ternary representation as $i$. Therefore they have the same middleness. According to the  propagation rule, the new cross-differences are decreased by a factor of $3$; this agrees with the formula, since the value of $m$ does not change while the value of $n$ increases by 1.
\end{proof}

The description of steeples that we inferred from Figure~\ref{fig:steeples} now follows as a corollary. 

\begin{corollary}
The steeple in an odd row is equal to the previous steeple times 3. The middle third of the steeple in an even row is equal to the previous steeple times 3, while the first and last third are all threes.
\end{corollary}

Each steeple is surrounded by ones, and every row consists of steeples of various heights together with ones; ones become new steeples in future rows. The next corollary describes how a new set of ones is generated from a steeple in an even row.

\begin{corollary}
In an odd row $n=2k-1$, if the middleness of  $i$ is $k+1$, then $C_n(i) = 1.$
\end{corollary}

\begin{proof}
These values are propagated from the set of threes in the previous steeple.
\end{proof}

\subsection{Recursive description}

We are now ready to describe the recursive construction of rows.

\begin{theorem}\label{thm:recursive}
Consider row $C_n$ and $m \leq \lceil n/2 \rceil$. The part of the row that has middleness $m$ consists of $3^{m-1}
$ copies of the row $C_{n+1-2m}$.
\end{theorem}

\begin{proof}
This part of the row is propagated from a series of $3^{m-1}
$ consecutive ones in row $C_{2m - 1}$, which are in turn generated from the steeple in row $C_{2m-2}$.
\end{proof}

For example, consider $C_4$. The first and last third of this row are copies of $C_3$, and the middle third is divided into three parts of equal size. The first and the last part are three copies of $C_1$. Finally, the middle part of the middle third is the steeple.

The next odd-indexed row has the same description as the previous row with shifted indices. For each new odd row the description goes deeper into the past.

\subsection{Counts}

In this section we count how many times each value appears in row $C_n$. We start by describing all the peaks. 

\begin{lemma}
All the peaks equal to $3^m$ in $C_n$, $m> 1$, are obtained as the propagation of the peaks $3^{m-1}$ at level $C_{n-1}$. All the peaks equal to 3 in $C_n$ are a propagation of all the terms of value 1 in $C_{n-1}$.
\end{lemma}

\begin{proof}
Peaks can only be achieved through the propagation of the first type: $V$ to $(V,3V,V)$.
\end{proof}

We would like to count how many of ones, threes and so on are there in the lists of cross-differences. For this we define two sequences: $a(n)= 3^n-(-1)^n$ and $b(n)=0^n$. The sequence $a(n)$ is sequence A105723 in the OEIS \cite{oeis}.  Starting from index zero the sequence $a(n)$ looks like:  0, 4, 8, 28, 80, 244, 728, 2188, and so on. The sequence $b(n)$ is the characteristic function of 0: sequence A7 in the OEIS.

The following theorem describes the counts in terms of these two sequences. The cross-differences in row $n$ are equal to $3^k$, where $0 \leq k \leq n$. 

\begin{theorem}\label{thm:counts}
In  row $n$, the number of cross-differences that are equal to $1$ is $b(n) + a(n)/2$. For $k > 0$, the number of cross-differences equal to $3^k$ is $b(n-k) + a(n - k)$: the number of peaks of value $3^k$ is $b(n-k) + a(n-k)/2$, while the number of non-peaks of value $3^k$ is $a(n-k)/2$.
\end{theorem}

\begin{proof}
The proof is by induction on $n$. For the base of induction consider $n$ equal to zero or one. In row zero, the number of cross-differences equal to 1 is $1 = b(0) + a(0)/2$. In row one, the number of cross-differences equal to 1 is $2 = a(1)/2$ and the number of 3-peaks is $1= b(0) + a(0)/2$. There are no other possible values for $k$. Let us now assume the statement is true up to row $i$.

The number of ones in row $i+1$ is twice the number of ones in the previous row plus three times the number of non-peak $3$s in the previous row. Thus it is equal to $2(b(i) + a(i)/2) + 3a(i-1)/2$. For $i \geq 1$ it iss not hard to check that $a(i+1)= 2a(i)+3a(i-1)$; indeed, $3^{i+1} - (-1)^{i+1} = 2(3^i - (-1)^i) + 3(3^{i-1} - (-1)^{i-1})$. Hence the claim holds when $k = 0$. 

In general, the number of non-peaks of value $3^k$ in row $i + 1$, where $k > 0$ is twice the number of peaks of the same value in the previous row plus three times the number of non-peaks of the next value. The calculation is the same as above.

The number of peaks of value $3^k$ in row $i + 1$, for $k > 0$, is the number of peaks equal to $3^{k-1}$ in row $i$, or the total number of ones if $k=1$. Thus  the number of peaks of value $3^k$ in row $i+1$ is $b(i-k+1) + a(i-k+1)/2$. This proves our statement for peaks.
\end{proof}

Hence the number of cross-differences that equal to 1 as a function of the row number $n$ is $1$ when $n = 0$ and $a(n)/2$ otherwise: 1, 2, 4, 14, 40, 122, 364, and so on. This is sequence A152011 in the OEIS \cite{oeis}. The number of cross-differences that are equal to  $3^k$ in row $n$ is $b(n-k) + a(n-k)$.

We note that the first time a particular value appears it occurs exactly once in the middle of its row. Afterwards, each non-one value is split evenly between  peaks and non-peaks.

\subsection{Ternary representation}

In this section we give a recursive method of computing the cross-difference at any index $i$ in terms of its ternary representation. Recall that since the first third of $C_{n}$ is a copy of $C_{n - 1}$, we can interpret the cross-differences as a single infinite sequence. Thus the value of the cross-difference depends only on the index $i$ (and not the row number). 

Suppose the ternary representation of $i$ has $n$ digits, and let $m = m(i)$ be the middleness of $i$. If $m > \lceil n/2 \rceil$, then we can apply the steeple formula in Lemma~\ref{thm:mainsteeple} to calculate the cross-difference explicitly as $3^{2m - n - 1}$. If $m \le \lceil n/2 \rceil$, then we conclude from Theorem~\ref{thm:recursive} that the range of indices with the same first $m$ digits as $i$ corresponds to $3^{m - 1}$ consecutive copies of $C_{n + 1 - 2m}$. Then the cross-difference at index $i$ is the same as the cross-difference at index $i'$, where $i'$ is obtained from $i$ by removing the $m + (m - 1) = 2m - 1$ leading ternary digits.

\begin{itemize}
\item If the number of leading ones is less than half of the total length, remove these ones and the same number of digits after that, plus one extra digit. Continue recursively.

\item  If the number of ones is not less than a half, then we are in a steeple and should use the steeple formula.
\end{itemize}
 
It is easy to see that only the positions of the $1$s in the ternary representation of $i$ affect the recursive procedure, so we may freely interchange $0$s and $2$s. Thus we have the following lemma:

\begin{lemma}
The value at index $i$ is the same as at index $j$ if $i$ and $j$ have $1$s in exactly the same positions in their ternary representations. 
\end{lemma}

\subsection{Particular values}

Let us look at the specific case when the cross-difference is one. From Theorem~\ref{thm:counts} we know that the number of ones in row $n$ is $(3^n-(-1)^n)/2$, so asymptotically half of the values are 1.

On the other hand, this set contains all the numbers without ones in their ternary representation. The latter set tends to the Cantor set, which has zero density --- our set is much bigger. Compare this to the no-reduction case, in which the cross-differences are one precisely on the Cantor set. To make the contrast clear, we draw the set of unit cross-differences in  Figure~\ref{fig:ones}; the format is the same as in Figure~\ref{fig:CantorSet}. 

\begin{figure}[htbp]
\centering
\includegraphics[scale=0.7]{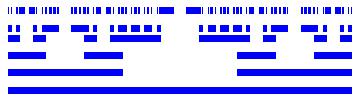}
\caption{The representation of ones}\label{fig:ones}
\end{figure}

\section{Quasi-fractals in real life}\label{sec:reallife}

Part of why fractals are such a popular object of study is their ubiquity in real life.  Famous examples of naturally occurring fractals include Romanesco broccoli, ammonite sutures, mountain ranges, and ferns. Surprisingly, the quasi-fractals that are discussed in the paper can be found in real life as well. The examples we present here are man-made, however. 

The contour in Figure~\ref{fig:4Red} reminded the second author of her alma mater, Moscow State University. Figure~\ref{fig:MGU} shows the front profile of the main building. Note how the center tower is not simply a scaled copy of the side towers, but is rather more elaborate -- it resembles a steeple. On the fringes of the center tower we can see smaller sub-towers which parallel Theorem \ref{thm:recursive}. The large flat portions mimic long runs of $1$s in the sequence of cross-differences. 

\begin{figure}[htbp]
\centering
\includegraphics[scale=0.05]{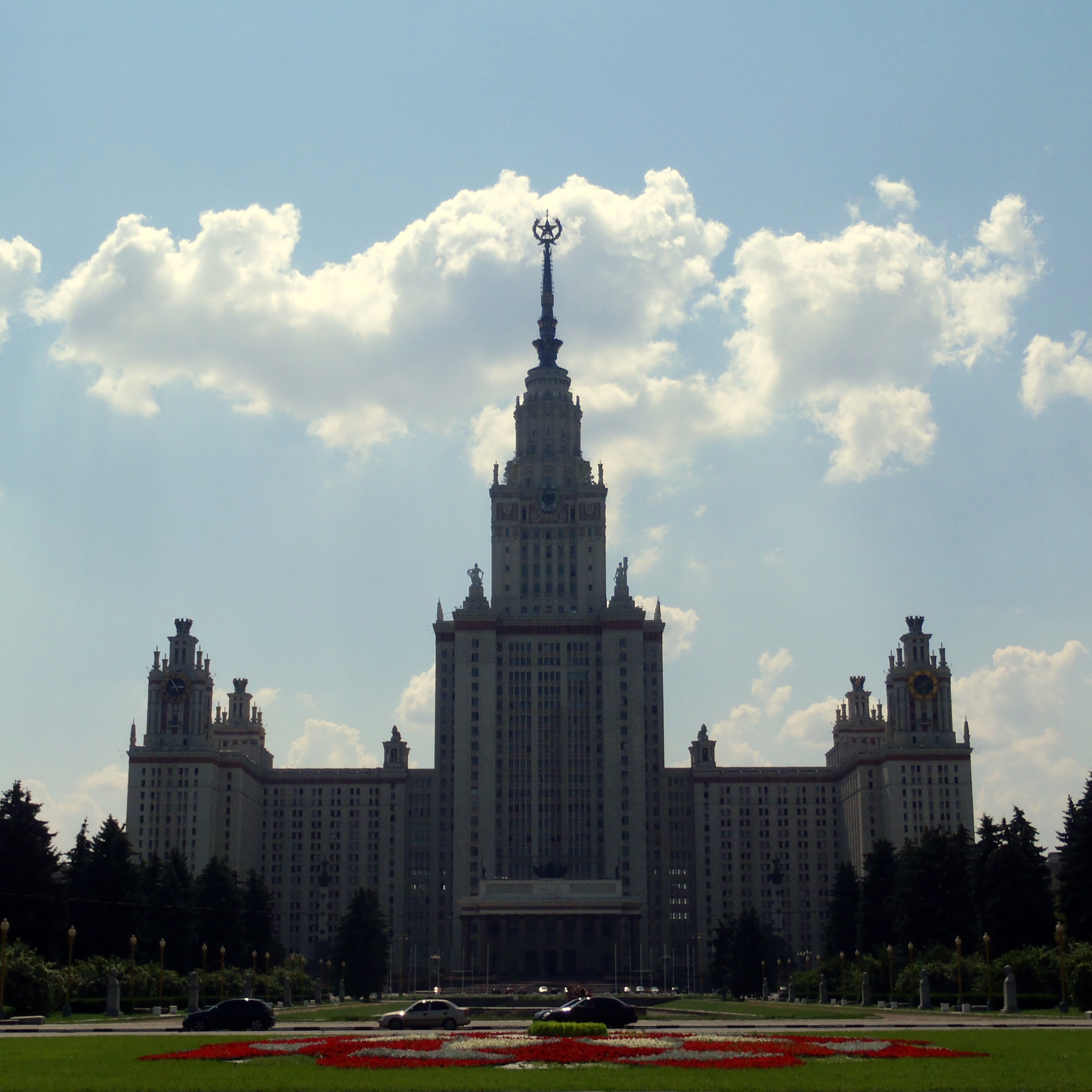}
\caption{Moscow State University \cite{MSU}}\label{fig:MGU}
\end{figure}

Quasi-fractals also appear in jewelry. In Figure~\ref{fig:necklace} below, we can see the motif of ``identical left and right part (earrings), with a similar but more complicated middle part (the pendant)'' once again. This holds recursively within the necklace as well. 

\begin{figure}[htbp]
\centering
\includegraphics[scale=0.5]{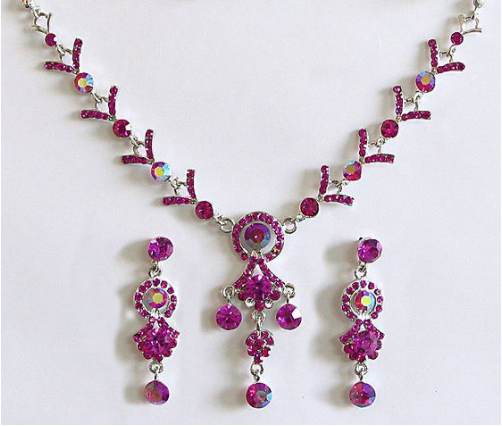}
\caption{Magenta Stone Studded Necklace and Earrings \cite{Necklace}}\label{fig:necklace}
\end{figure}

\section{Acknowledgements}

We would like to thank Prof. James Propp (UMass) for suggesting the project and discussing it with us.

\end{document}